\documentclass[10pt, a4paper, leqno]{amsart}
\newcounter{minutes}
\divide\time by 60
\newcounter{hours}
\multiply\time by 60 \addtocounter{minutes}{-\time}

\usepackage{amssymb}
\usepackage{hyperref}
\usepackage{graphicx}
\date{}

\usepackage{amsfonts}
\usepackage{amsthm}
\usepackage{amsmath}
\usepackage{graphicx}
\usepackage{amscd}
\usepackage{color}
\usepackage{epstopdf}

\title[{Inequalities for generalized inverse trigonometric functions}]
{Functional inequalities for generalized inverse trigonometric and hyperbolic functions}

\author[\'A. Baricz]{\'Arp\'ad Baricz$^{\dagger}$}
\address{Department of Economics, Babe\c{s}-Bolyai University, 400591 Cluj-Napoca, Romania}
\email{bariczocsi@yahoo.com}

\author[B.A. Bhayo]{Barkat Ali Bhayo}
\address{Department of Mathematical Information Technology, University of Jyv\"askyl\"a,
40014 Jyv\"askyl\"a, Finland} \email{bhayo.barkat@gmail.com}

\author[T.K. Pog\'any]{Tibor K. Pog\'any$^{\ddagger}$}
\address{Faculty of Maritime Studies, University of Rijeka, 51000 Rijeka, Croatia} \email{poganj@pfri.hr}

\thanks{$^{\dagger}$The work of \'A. Baricz was supported by a research grant of the Babe\c{s}-Bolyai University for young researchers, project number GTC\underline{ }34021. \,
$^{\ddagger}$Corresponding author}

\numberwithin{equation}{section}

\theoremstyle{plain}
\newtheorem{theorem}{Theorem}
\newtheorem{lemma}{Lemma}
\newtheorem*{remark}{Remark}

\begin{document}
\allowdisplaybreaks

\begin{abstract}
Various miscellaneous functional inequalities are deduced for the so-called generalized inverse trigonometric and hyperbolic functions. For instance, functional inequalities for sums, difference and quotient of generalized inverse trigonometric and hyperbolic functions are given, as well as some Gr\"unbaum inequalities with the aid of the classical Bernoulli inequality. Moreover, by means of certain already derived bounds, bilateral bounding inequalities are obtained for the generalized hypergeometric ${}_3F_2$ Clausen function.\\


\end{abstract}

\def\thefootnote{}
\footnotetext{ \texttt{\tiny File:~\jobname .tex,
          printed: \number\year-0\number\month-\number\day,
          \thehours.\ifnum\theminutes<10{0}\fi\theminutes}
} \makeatletter\def\thefootnote{\@arabic\c@footnote}\makeatother

\maketitle

\section{\bf Introduction and preliminary results}

For given complex numbers $a,b$ and $c$ with $c\neq0,-1,-2,\ldots$,
the \emph{Gaussian hypergeometric function} ${}_2F_1$ is the
analytic continuation to the slit place $\mathbb{C}\setminus[1,\infty)$ of the series
   $$F\left(a,b;c;z\right) = {}_2F_1\left(a,b;c;z\right)=\sum_{n\geq0}\frac{(a,n)(b,n)}
                             {(c,n)}\frac{z^n}{n!},\qquad |z|<1.$$
Here $(a,n)$ is the Pochhammer symbol (rising factorial) $(\cdot,n):\mathbb{C}\to \mathbb{C}$, defined by
   $$(z,n) = \frac{\Gamma(z+n)}{\Gamma(z)} = \prod_{i=1}^n(z+i-1)$$
for $n\in\mathbb{Z}$, see \cite{as}. Special functions, such the classical \emph{gamma function} $\Gamma$, the {\it digamma function} $\psi$ and the \emph{beta function}  $B(\cdot,\cdot)$ have close relation with hypergeometric function. These functions for $x,y>0$ are defined by
   $$\Gamma(x) = \int^\infty_0 e^{-t}t^{x-1}\,dt,\quad
       \psi(x) = \frac{\Gamma'(x)}{\Gamma(x)},\quad
        B(x,y) = \frac{\Gamma(x)\Gamma(y)}{\Gamma(x+y)},$$
respectively.

The eigenfunction $\sin_p$ of the of the so-called one-dimensional $p$-Laplacian problem
\cite{dm}
$$-\Delta_p u=-\left(|u'|^{p-2}u'\right)'
=\lambda|u|^{p-2}u,\,u(0)=u(1)=0,\ \ \ p>1,$$
is the inverse function of $F:(0,1)\to \left(0,\frac{\pi_p}{2}\right)$, defined as
$$F(x)={\rm arcsin}_p(x)=\int^x_0(1-t^p)^{-\frac{1}{p}}dt,$$
where
$$\pi_p=\frac{2}{p}\int^1_0(1-s)^{-\frac{1}{p}}s^{\frac{1}{p}-1}ds=\frac{2}
{p}\,B\left(1-\frac{1}{p},\frac{1}{p}\right)=\frac{2 \pi}{p\,\sin\left(\frac{\pi}{p}\right)}\,.$$
The function ${\rm arcsin}_p$ is called as the generalized inverse sine function, and coincides with usual inverse sine function
for $p=2$. Similarly, the other generalized inverse trigonometric and hyperbolic functions
${\rm arccos}_p:(0,1)\to \left(0,\frac{\pi}{2}\right),\,{\rm arctan}_p:(0,1)\to (0,b_p),\,{\rm arcsinh}_p:(0,1)\to(0,c_p),\,
{\rm arctanh}_p:(0,1)\to (0,\infty)$, where
$$\,b_p=\frac{1}{2p}\left(\psi\left(\frac{1+p}{2p}\right)
-\psi\left(\frac{1}{2p}\right)\right)=2^{-\frac{1}{p}}
F\left(\frac{1}{p},\frac{1}{p};1+\frac{1}{p};\frac{1}{2}\right),\ c_p=\left(\frac{1}{2}\right)^{\frac{1}{p}}F\left(1,\frac{1}{p};1+\frac{1}{p},\frac{1}{2}\right),$$
are defined as follows
\begin{eqnarray*}
{\rm arccos}_p(x)&=&\int^{(1-x^p)^{\frac{1}{p}}}_0(1-t^p)^{-\frac{1}{p}}dt,\\
{\rm arctan}_p(x)&=&\int^x_0(1+t^p)^{-1}dt,\\
{\rm arcsinh}_p(x)&=&\int^x_0(1+t^p)^{-\frac{1}{p}}dt,\\
{\rm arctanh}_p(x)&=&\int^x_0(1-t^p)^{-1}dt.
\end{eqnarray*}

These functions are the generalizations of the usual elementary inverse trigonometric and hyperbolic functions and are the inverse of the so-called generalized trigonometric and hyperbolic functions, introduced by P. Lindqvist \cite{l}, see also \cite{be,egl,lp,ti} for more details. Recently, there has been a vivid interest on the generalized inverse trigonometric and hyperbolic functions, we refer to the papers \cite{bbv, bbk, bsand,bvjapprox,bvarxiv,jq,kvz} and to the references therein. In this paper we make a contribution to the subject by showing various miscellaneous functional inequalities for the generalized inverse trigonometric and hyperbolic functions. The paper is organized as follows. In this section we list some preliminary results which will be used in the sequel. Section 2 contains functional inequalities for sums, difference and quotient of generalized inverse trigonometric and hyperbolic functions, as well as some Gr\"unbaum inequalities with the aid of the classical Bernoulli inequality. In section 3 we obtain some lower and upper bounds for the generalized hypergeometric ${}_3F_2$ Clausen function by using certain already derived bounds for generalized inverse trigonometric and hyperbolic functions. Finally, in section 4 we give a comparison of the bounds of this paper with the known bounds in the literature.

For the expression of the above generalized inverse trigonometric and hyperbolic functions in terms of hypergeometric functions, see the following lemma.

\begin{lemma}\label{hypform}
For $p>0$ and $x\in(0,1)$, we have
\begin{eqnarray*}
{\rm arcsin}_p(x)&=& x\,F\left(\frac{1}{p},\frac{1}{p};1+\frac{1}{p};x^p\right),\\
{\rm arctan}_p(x)&=& x
F\left(1,\frac{1}{p};1+\frac{1}{p};-x^p\right)=\left(\frac{x^p}{1+x^p}\right)^{\frac{1}{p}}
F\left(\frac{1}{p},\frac{1}{p};1+\frac{1}{p};\frac{x^p}{1+x^p}\right),\\
{\rm arcsinh}_p(x)&=&
xF\left(\frac{1}{p}\,,\frac{1}{p};1+\frac{1}{p};-x^p\right)=\left(\frac{x^p}{1+x^p}\right)^{\frac{1}{p}}
F\left(1,\,\frac{1}{p};\,1+\frac{1}{p};\,\frac{x^p}{1+x^p}\right),\\
{\rm arctanh}_p(x)&=&xF\left
(1\,,\frac{1}{p};1+\frac{1}{p};x^p\right).
\end{eqnarray*}
\end{lemma}

\begin{proof}[\bf Proof]
The proof of all equalities is similar, here we only give the proof of the last equality, which can be written as
   $$\int_0^x{\frac{1}{1-t^{\frac{1}{m}}}dt}=xF\left(1,m;1+m;x^{\frac{1}{m}}\right),$$
setting $\frac{1}{p} = m>0.$
By using the differentiation formula \cite[15.2.1]{as}
\begin{equation}\label{dF}
\frac{d}{dz}F\left(a,b;c;z\right)=\frac{ab}{c}F\left(a+1,b+1;c+1;z\right)
\end{equation}
we get
\begin{equation}\label{deri}
\frac{d}{dx}\left[xF\left(1,m;1+m;x^{\frac{1}{m}}\right)\right]=F\left(1,m;1+m;x^{\frac{1}{m}}\right)+\frac{x^{\frac{1}{m}}}{1+m}F\left(2,1+m;2+m;x^{\frac{1}{m}}\right).
\end{equation}
The right-hand side can be simplified by using the Gauss relationship for contiguous hypergeometric functions, namely
\begin{equation}\label{gc2}
zF\left(a+1,b+1;c+1;z\right)=\frac{c}{a-b}\left(F\left(a,b+1;c;z\right)
-F\left(a+1,b;c;z\right)\right),
\end{equation}
\begin{equation}\label{gc1}
(b-a)F\left(a,b;c;z\right)+aF\left(a+1,b;c;z\right)=bF\left(a,b+1;c;z\right),
\end{equation}	
see \cite[p. 58]{em} and \cite[(15.5.12)]{olbc}. After simplification and utilizing
$$F\left(a,b;b;z\right)=(1-z)^{-a}$$
see \cite[(15.4.6)]{olbc}, the right-hand side of (\ref{deri}) simplifies to
$\left(1-x^{\frac{1}{m}}\right)^{-1}$. This implies the proof.
\end{proof}

\begin{remark}
{\em Alternatively, the following derivation procedure can be considered. Since $x \in (0,1)$, the
integrand can be presented by the associated geometric series, which is termwise integrable for all $p>1$
   \[ \int_0^x \frac1{1-t^p}\,dt = \sum_{n \geq 0} \int_0^x t^{np}\,dt
              = \sum_{n \geq 0} \frac{x^{np+1}}{np+1}\, .\]
Since
   \begin{equation} \label{FX}
      \frac1{np+1} = \frac1p\, \frac1{\frac1p+n}
                   = \frac1p\, \frac{\Gamma\left(\frac1p+n\right)}{\Gamma\left(\frac1p+1+n\right)}
                   = \frac{\left(\frac1p,n\right)}{\left(\frac1p+1,n\right)},
   \end{equation}
we conclude
   \[ {\rm arctanh}_p(x) = x\,\sum_{n \geq 0} \frac{(1,n)\,\left(\frac1p,n\right)}{\left(\frac1p+1,n\right)}\,
                           \frac{(x^p)^m}{n!}\, ,\]
which proves the hypergeometric representation of generalized inverse hyperbolic tangent.

Applying the same procedure we also have
   \[ \int_0^x \frac1{(1-t^p)^{\frac1p}}\,dt = \sum_{n \geq 0}(-1)^n \binom{-\frac1p}n
      \int_0^x t^{np}\,dt = \sum_{n \geq 0} (-1)^n \binom{-\frac1p}n \frac{x^{np+1}}{np+1}\, .\]
Bearing in mind the fact
   \[ (-1)^n \binom{-\frac1p}n = (-1)^n \frac{\left(-\frac1p\right)\left(-\frac1p-1\right) \cdots \left(-\frac1p-n+1\right)}{n!}
       = \frac{\left(\frac1p,n\right)}{n!} \]
in conjunction with \eqref{FX}, we arrive at
   \[ \int_0^x \frac1{(1-t^p)^{\frac1p}}\,dt = x \sum_{n \geq 0}
               \frac{\left(\frac1p,n\right)\left(\frac1p,nt\right)}{\left(\frac1p+1,n\right)}\, \frac{(x^p)^n}{n!}\, ,\]
which confirms the hypergeometric expression for the $\arcsin_p$ function.}
\end{remark}
For the following lemma see \cite[Theorem 1.52]{avvb}.
\begin{lemma}\label{thm1.52}
For $a,b>0$, the function $f,$ defined by
$$f(x)=\displaystyle\frac{1-F\left(a,b;a+b;x\right)}{\log(1-x)},$$
is strictly increasing from $(0,1)$ onto $\displaystyle\left(\frac{ab}{a+b},\frac{1}{B(a,b)}\right)$.
\end{lemma}

The next result will be useful to prove some Gr\"unbaum type inequalities for generalized inverse trigonometric functions. For more details see \cite[Lemma 2.1]{bari}.

\begin{lemma}\label{lem00} Let us consider the function $f : (a,\infty)\to \mathbb{R}$, where $a\geq 0$. If the function $g$, defined by $g(x) =\frac{1}{x}\left[f(x) - 1\right],$ is increasing on $(a,\infty)$, then for the function $h$, defined by $h(x) = f(x^2)$, we
have the following Gr\"unbaum-type inequality
\begin{equation}\label{lemm01}
1 + h(z) \geq  h(x) + h(y),
\end{equation}
where $x, y \geq a$ and $z^2 = x^2+y^2$. If the function $g$ is decreasing, then inequality
(\ref{lemm01}) is reversed.
\end{lemma}

\section{\bf Inequalities for generalized inverse trigonometric and hyperbolic functions}
\setcounter{equation}{0}

Our first main result is the following theorem.

\begin{theorem}\label{lem1} For $p>1$ and $x\in(0,1)$, we have the inequalities
\begin{equation}\label{ineq1}\left(\frac{{\rm arcsin}_p\,x}{x}\right)^{p} <\frac{{\rm arctanh}_p\,x}{x}<
\frac{{\rm arcsin}_p\,x}{x\sqrt[p]{1-x^p}}\,.\end{equation}
\begin{equation}\label{ineqqwe}\left(\frac{x^p}{1+x^p}\right)^{\frac{1}{p}-1}\left({\rm arcsin}_p\left(\frac{x^p}{1+x^p}\right)\right)^p
< {\rm arcsinh}_p(x)<(1+x^p)^{\frac{1}{p}}{\rm arcsin}_p\,\left(\left(\frac{x^p}{1+x^p}\right)^{\frac{1}{p}}\right).\end{equation}
Moreover, the second inequality in \eqref{ineq1} holds true for $p>0$ and $x\in(0,1).$
\end{theorem}

\begin{proof}[\bf Proof]
We know (see \cite{ne4}) that for $b,c>0$ and $x\in(0,1)$ the function $a\mapsto F(a,b;c;x)$ is logarithmically convex on $(0,\infty).$ This implies that $a\mapsto \left[\partial F(a,b;c;x)/\partial a\right]/F(a,b;c;x)$ is increasing on $(0,\infty)$. By using the monotone form of l'Hospital's rule we obtain that the function
$$a\mapsto \frac{\log F(a,b;c;x)}{a}=\frac{\log F(a,b;c;x)-\lim\limits_{a\to 0}\log F(a,b;c;x) }{a-0}$$
is also increasing on $(0,\infty)$ for $b,c>0$ and $x\in(0,1).$ Consequently, the function $a\mapsto F(a,b;c;x)^{1/a}$ is also increasing on $(0,\infty)$ for $b,c>0$ and $x\in(0,1).$ Thus, we get
$$\frac{x\,F\left(1,\frac{1}{p};1+\frac{1}{p};x^p\right)}{x}>
\left(\frac{x\,F\left(\frac{1}{p},\frac{1}{p};1+\frac{1}{p};x^p\right)}{x}\right)^{p},$$
and the first inequality in \eqref{ineq1} follows. Similarly, if use again the fact that the function $a\mapsto F(a,b;c;x)^{1/a}$ is increasing on $(0,\infty)$ for $b,c>0$ and $x\in(0,1),$ we obtain that for all $p>1$ and $x\in(0,1)$ the next inequality follows
$$\left(\frac{x^p}{1+x^p}\right)^{\frac{1}{p}}F\left(1,\frac{1}{p};1+\frac{1}{p};\frac{x^p}{1+x^p}\right)
>\left(\frac{x^p}{1+x^p}\right)^{\frac{1}{p}}\left(F\left(\frac{1}{p},\frac{1}{p};1+\frac{1}{p};\frac{x^p}{1+x^p}\right)\right)^p,$$
which is equivalent to the left-hand side of \eqref{ineqqwe}.

For the second inequality in \eqref{ineq1}, we consider the function $f:(0,1)\to(0,1),$ defined by $f(y)=y(1-\log y),$
which is increasing. In view of the above function we obtain that the function $x\mapsto (1-x^p)^{\frac{1}{p}}(1-\log((1-x^p)^{\frac{1}{p}}))$
is decreasing on $(0,1)$. It follows that
$$0<1+\frac{x^p}{p(1+p)}-1<1+\frac{x^p}{p(1+p)}-(1-x^p)^{\frac{1}{p}}\left(1-\log\left((1-x^p)^{\frac{1}{p}}\right)\right),$$
which is equivalent to
\begin{equation}\label{logo}(1-x^p)^{\frac{1}{p}}<\frac{1+\frac{x^p}{p(1+p)}}{1-\frac{1}{p}\log(1-x^p)}.\end{equation}
On the other hand, for $p>1$ and $x\in(0,1)$ we have \cite[Theorem 1.1]{bvjapprox}
\begin{equation}\label{ineqarsin}\left(1+\frac{x^p}{p(1+p)}\right)x<
{\rm arcsin}_p\,x.\end{equation}
Moreover, by Lemma \ref{thm1.52} the function
$$x\mapsto \frac{\displaystyle1-\frac{{\rm arctanh}_p(x)}{x}}{\log(1-x^p)}$$
is increasing from $(0,1)$ onto $\left(\frac{1}{p+1},\frac{1}{p}\right),$ and consequently for all $p>0$ and $x\in(0,1)$ we have
\begin{equation}\label{ineqartanh}
x\left(1-\frac{1}{1+p}\log(1-x^p)\right)< {\rm arctanh}_p\,x
< x\left(1-\frac{1}{p}\log(1-x^p)\right).
\end{equation}
Now, combining \eqref{ineqarsin} with the right-hand side of \eqref{ineqartanh} we obtain the inequality
$$(1-x^p)^{\frac{1}{p}}<\frac{1+\frac{x^p}{p(1+p)}}{1-\frac{1}{p}\log(1-x^p)}
<F\left(\frac{1}{p},\frac{1}{p};1+\frac{1}{p};x^p\right)F\left(1,\frac{1}{p};1+\frac{1}{p};x^p\right)^{-1}\,,$$
which is equivalent to
\begin{equation}\label{wert}F\left(1,\frac{1}{p};1+\frac{1}{p};x^p\right)<F\left(\frac{1}{p},\frac{1}{p};1+\frac{1}{p};x^p\right)(1-x^p)^{-\frac{1}{p}}.\end{equation}
With this we proved the second inequality of \eqref{ineq1}. Now, for the second inequality of \eqref{ineqqwe} we note that from Lemma \ref{hypform}
we get that
$${\rm arcsinh}_p(x)={\rm arctanh}_p\left(\left(\frac{x^p}{1+x^p}\right)^{\frac{1}{p}}\right).$$
Thus, if we change $x$ to $\left(\frac{x^p}{1+x^p}\right)^{\frac{1}{p}}$ in \eqref{wert} we get the second inequality in \eqref{ineqqwe}.
\end{proof}

The next theorem gives some monotonicity results, as well as some lower and upper bounds for the functions ${\rm arcsinh}_p$ and ${\rm arctan}_p.$

\begin{theorem}\label{thmbar}
The functions
$$x\mapsto \frac{xF\left(\frac{1}{p},1+\frac{1}{p};2+\frac{1}{p};-x^p\right)}{{\rm arcsinh}_p(x)}, \qquad x\mapsto \frac{xF\left(2,\frac{1}{p};2+\frac{1}{p};-x^p\right)}{{\rm arctan}_p(x)}$$
are decreasing on $(0,1)$ for all $p>0,$ while the function
$$x\mapsto\frac{{\rm arcsinh}_p(x)}{xF\left(-1+\frac{1}{p},\frac{1}{p};\frac{1}{p};-x^p\right)}$$
is also decreasing on $(0,1)$ for all $p\in(0,1).$ Consequently, for all $p\in(0,1)$ and $x\in(0,1)$ we have
\begin{equation}\label{bounds}
xF\left(\frac{1}{p},1+\frac{1}{p};2+\frac{1}{p};-x^p\right)<{\rm arcsinh}_p(x)<xF\left(-1+\frac{1}{p},\frac{1}{p};\frac{1}{p};-x^p\right).
\end{equation}
Moreover, the left-hand side of \eqref{bounds} holds true for all $x\in(0,1)$ and $p>0,$ as well as the inequality
$${xF\left(2,\frac{1}{p};2+\frac{1}{p};-x^p\right)}<{{\rm arctan}_p(x)}.$$
\end{theorem}

\begin{proof}[\bf Proof]
We shall use the following result of Belevitch \cite[p. 1032]{belevitch}
$$\frac{F(\alpha+1,\beta;\gamma+1;-z)}{F(\alpha,\beta;\gamma;-z)}=\left\{\begin{array}{ll}\displaystyle S_{\alpha,\beta,\gamma}(z),& \alpha\leq\beta\\  \displaystyle S_{\alpha,\beta,\gamma}(z)+\frac{\gamma(\alpha-\beta)}{2\alpha(\gamma-\beta)},& \alpha>\beta\end{array},\right.$$
where
$$S_{\alpha,\beta,\gamma}(z)=\frac{\Gamma(\gamma)\Gamma(\gamma+1)}{\Gamma(\alpha+1)\Gamma(\beta)\Gamma(\gamma-\alpha)\Gamma(\gamma-\beta+1)}\int_0^1\frac{t^{\alpha+\beta-1}(1-t)^{\gamma-\alpha-\beta}}{(1+zt)\left|F(\alpha,\beta;\gamma;t^{-1})\right|^2}dt$$
and
$0\leq\alpha\leq\gamma,$ $0\leq \beta\leq\gamma,$ $\gamma\geq1,$ $z\in(0,1).$ By choosing $\alpha=\frac{1}{p},$ $\beta=\frac{1}{p},$ $\gamma=1+\frac{1}{p}$ and $z=x^p$ we get
$$\frac{xF\left(1+\frac{1}{p},\frac{1}{p};2+\frac{1}{p};-x^p\right)}{{\rm arcsinh}_p(x)}=\frac{1}{p}\left(1+\frac{1}{p}\right)
\displaystyle\int_0^1\frac{t^{\frac{2}{p}-1}(1-t)^{1-\frac{1}{p}}}{(1+x^pt)\left|F\left(\frac{1}{p},\frac{1}{p};1+\frac{1}{p};t^{-1}\right)\right|^2}dt.$$
Similarly, by choosing $\alpha=\frac{1}{p}-1,$ $\beta=\frac{1}{p},$ $\gamma=\frac{1}{p}$ and $z=x^p$ we get
$$\frac{{\rm arcsinh}_p(x)}{xF\left(-1+\frac{1}{p},\frac{1}{p};\frac{1}{p};-x^p\right)}=\frac{1}{p}\left(\frac{1}{p}-1\right)
\displaystyle\int_0^1\frac{t^{\frac{2}{p}-2}(1-t)^{1-\frac{1}{p}}}{(1+x^pt)\left|F\left(-1+\frac{1}{p},\frac{1}{p};\frac{1}{p};t^{-1}\right)\right|^2}dt.$$
Now, by choosing $\alpha=1,$ $\beta=\frac{1}{p},$ $\gamma=1+\frac{1}{p}$ and $z=x^p$ we get for $p\in(0,1]$
$$\frac{xF\left(2,\frac{1}{p};2+\frac{1}{p};-x^p\right)}{{\rm arctan}_p(x)}=\frac{1}{p^2}\left(1+\frac{1}{p}\right)
\displaystyle\int_0^1\frac{t^{\frac{1}{p}}}{(1+x^pt)\left|F\left(1,\frac{1}{p};1+\frac{1}{p};t^{-1}\right)\right|^2}dt,$$
and for $p>1$
$$\frac{xF\left(2,\frac{1}{p};2+\frac{1}{p};-x^p\right)}{{\rm arctan}_p(x)}=\frac{1}{2}\left(1-\frac{1}{p^2}\right)+\frac{1}{p^2}\left(1+\frac{1}{p}\right)
\displaystyle\int_0^1\frac{t^{\frac{1}{p}}}{(1+x^pt)\left|F\left(1,\frac{1}{p};1+\frac{1}{p};t^{-1}\right)\right|^2}dt.$$

Differentiating both sides of the above relations the monotonicity results follow. Moreover, computing the limit of the functions in zero, the claimed inequalities also follow from the monotonicity results.
\end{proof}

Direct application of Theorem \ref{thmbar} with the use of the following transformation formula
$$F(a,b;c;z)=(1-z)^{-a}F\left(b,c-a;c;-\frac{z}{1-z}\right)$$
gives the following inequalities
   \begin{equation}\label{ineq1}
      l_p(x) < {\rm arcsinh}_p(x) < u_p(x),
   \end{equation}
where
   \[ l_p(x) = \left(\frac{x^p}{1+x^p}\right)^{\frac{1}{p}}
               F\left(1,\frac{1}{p};2+\frac{1}{p};\frac{x^p}{1+x^p}\right),\]
and $u_p(x)=a_1(p)l_p(x)$, with
   \[ a_1(p) = \frac{F\left(1,\frac{1}{p};2+\frac{1}{p};\frac{1}{2}\right)}{F\left(1,\frac{1}{p};1+\frac{1}{p};\frac{1}{2}\right)}.\]
The next consequence of Theorem \ref{thmbar} is
   \begin{equation} \label{ineq2}
      \tilde{l}_p(x) < {\rm arctan}_p(x) < a_2(p)\tilde{l}_p(x),
   \end{equation}
where
   \[ \tilde{l}_p(x) = \left(\frac{\sqrt{x}}{1+x^p}\right)^{2}
             F\left(2,2;2+\frac{1}{p};\frac{x^p}{1+x^p}\right),\]
and $\tilde{u}_p(x)=a_2(p)\tilde{l}_p(x)$, with
   \[ a_2(p) = \frac{2^{2-\frac{1}{p}}F\left(\frac{1}{p},\frac{1}{p};1+\frac{1}{p};\frac{1}{2}\right)}{F\left(2,2;2+\frac{1}{p};\frac{1}{2}\right)}.\]
For the convenience of the reader, we recall the following inequalities again
from \cite{bvarxiv}.
   \begin{equation} \label{oldineq}
      \begin{aligned}
         m_p(x) = \frac{(p(1+p)(1+x^p)+x^p)x}{p(1+p)(1+x^p)^{1+\frac1p}}
                < {\arctan}_p(x) < \frac{2^pb_p\, x}{(1+x^p)^{\frac1p}} = M_p(x),\\
         t_p(x) = \frac{x(1+\frac1{1+p}\log(1+x^p))}{(1+x^p)^{\frac1p}}
                < {\rm arcsinh}_p(x) <
                \frac{x(1+\frac1p \log(1+x^p))}{(1+x^p)^{\frac1p}} = T_p(x),
      \end{aligned}
   \end{equation}
where $x\in(0,1)$ and $p>1.$ In order to compare the bounds in \eqref{ineq1} and \eqref{ineq2} with the corresponding bounds of the functions given in \eqref{oldineq}, we made some computer experiments. From graphics we have seen that our lower bound $\tilde{l}_p$ is not better than the known lower bound $m_p$ from \cite{bvarxiv}, however our new upper bound $\tilde{u}_p$ for $p$ greater than $1$ is better than the known upper bound $M_p.$ Accordingly, we remark that by newly derived bounds \eqref{bounds} we complement the parameter interval for until now not considered $p \in (0,1)$ concerning ${\rm arcsinh}_p$ function, studied until now for $p>1$.

Similarly, our new lower bound $l_p$ is not better than the known lower bound $t_p,$ however, the upper bound $u_p$ complements the known upper bound $T_p.$

For $a,b>0$ and $x,y\in(0,1)$, let
$$Q_g(x,y)=\frac{g(x)+g(y)}{g(x+y-xy)},\,D_g(x,y)=g(x)+g(y)-g(x+y-xy),$$
where
$$g(z)=z F\left(a,b;a+b;z\right).$$
We use also the notation $R(a,b)=-2c-\psi(a)-\psi(b),$
where $c$ is the\emph{ Euler-Mascheroni} constant
$$c=\lim_{n\to \infty}\left(\sum_{k=1}^{n}\frac{1}{k}-\log n\right)=0.577215{\ldots}.$$

Recently Simi\'c and Vuorinen studied the inequalities about the quotient and differences involving hypergeometric functions, see \cite{svu,svu2}, and they proved the following result: if $a,b>0$ such that $ab\leq1$ and $x,y\in(0,1)$, then the following inequalities hold
$$\frac{1}{B(a,b)}\leq Q_g(x,y)\leq B(a,b),$$
$$\frac{B(a,b)-1}{R(a,b)}\leq Q_g(x,y)\leq \frac{2R(a,b)}{B(a,b)-1},$$
$$0\leq D_g(x,y)\leq \displaystyle\frac{2R(a,b)+1}{B(a,b)}-1.$$

Now, let $g_p(z)=z F\left(a,b;a+b;z^p\right).$ If we replace in the proof of the above inequalities the function $g$ by $g_p,$ then we get the following theorem, which we affirm here without proof.

\begin{theorem}\label{thm1} If $p> 1$ and $x,y\in(0,1)$, then we have the following inequalities
$$\frac{1}{p}\leq \frac{{\rm arctanh}_p(x)+{\rm arctanh}_p(y)}{{\rm arctanh}_p(x+y-xy)}
\leq p,$$
$$\frac{p-1}{R\left(1,\frac{1}{p}\right)}\leq  \frac{{\rm arctanh}_p(x)+{\rm arctanh}_p(y)}{{\rm arctanh}_p(x+y-xy)}
\leq \frac{2R\left(1,\frac{1}{p}\right)}{p-1},$$
$$0\leq{\rm arctanh}_p(x)+{\rm arctanh}_p(y)-{\rm arctanh}_p(x+y-xy)
\leq \frac{2R\left(1,\frac{1}{p}\right)-1}{p}-1.$$
\end{theorem}

The next result gives another lower bounds for the sum ${\rm arctanh}_p(x)+{\rm arctanh}_p(y).$

\begin{theorem} The function $f:(0,\infty) \to (0,\infty),$ defined by
$f(x)={\rm arctanh}_p\left(\frac{1}{\cosh(x)}\right),$ is strictly decreasing and convex for $p\in[1,2].$ Consequently, for all
$x,y\in(0,1)$ and $p\in[1,2]$ we have
\begin{equation}\label{artancosh1}
{\rm arctanh}_p(x)+{\rm arctanh}_p(y)> 2\cdot{\rm arctanh}_p\left(\sqrt{\frac{2xy}{1+xy+x'y'}}\right),
\end{equation}
where $z'=\sqrt{1-z^2}$. Moreover, if $p>0$, then the function $g:(1,\infty)\to(0,\infty),$ defined by
$g(x)={\rm arctanh}_p\left(\frac{1}{x}\right),$ is strictly decreasing and convex. Consequently, for all $x,y\in(0,1)$ and $p>0$ we obtain
\begin{equation}\label{artancosh2}
{\rm arctanh}_p(x)+{\rm arctanh}_p(y)> 2\cdot{\rm arctanh}_p\left(\frac{2xy}{x+y}\right).
\end{equation}
\end{theorem}

\begin{proof}[\bf Proof] Simple calculation gives for all $x>0$
$$f'(x)=-\frac{{\rm sech}(x)\tanh(x)}{1-{\rm sech}(x)^p}<0.$$
Clearly ${\rm sech}$ is decreasing, and thus for the convexity of $f$ it is enough to prove that the function $g:(0,\infty)\to(0,\infty),$ defined by
$g(x)=\tanh(x)/(1-{\rm sech}(x)^p),$ is decreasing. Differentiating with respect to $x$ we get for $p\in[1,2]$ and $x>0$
\begin{align*}
g'(x)&=-\frac{{\rm sech}^{p+2}(x)}{(1-{\rm sech}^p(x))^2}(1+p\cosh^2(x)-p-\cosh^p(x))\\
&<-\frac{{\rm sech}^{p+2}(x)}{(1-{\rm sech}^p(x))^2}(1+p\cosh^2(x)-p-\cosh^2(x))\\
&<-\frac{{\rm sech}^{p+2}(x)}{(1-{\rm sech}^p(x))^2}((p-1)(\cosh^2(x)-1))<0.
\end{align*}
Hence $g$ is decreasing, and this in turn implies that $f'$ is increasing for $p\in[1,2]$. This implies that $f$ is convex if $p\in[1,2]$.
Hence for all $r,s>0$ and $p\in[1,2]$ we get
\begin{equation}\label{artancosh}
\frac{1}{2}\left({\rm arctanh}_p\left(\frac{1}{\cosh(r)}\right)+{\rm arctanh}_p\left(\frac{1}{\cosh(s)}\right)\right)>{\rm arctanh}_p\left(\frac{1}{\cosh\left(\frac{r+s}{2}\right)}\right).
\end{equation}
Setting $x=1/\cosh(r),$ $y=1/\cosh(s)$ and using the identity
$$\cosh^2\left(\frac{r+s}{2}\right)=\frac{1+xy+x'y'}{2xy}$$ and inequality \eqref{artancosh}, we get the required inequality \eqref{artancosh1}.

Since for all $x>1$ and $p>0$ we have $$g'(x)=\frac{x^{p-2}}{1-x^p}<0 \ \ \ \mbox{and}\ \ \ g''(x)=\frac{(p-2+2x^p)x^{p-3}}{(1-x^p)^2}>0,$$
we get that indeed $g$ is strictly decreasing and convex for all $p>0.$ Consequently, for all $r,s>1$ and $p>0$ we obtain the following inequality
$$2\cdot {\rm arctanh}_p\left(\frac{2}{r+s}\right)< {\rm arctanh}_p\left(\frac{1}{r}\right)
+{\rm arctanh}_p\left(\frac{1}{s}\right),$$
which is equivalent to \eqref{artancosh2}.
\end{proof}

Now, we focus on Gr\"unbaum type inequalities for generalized inverse trigonometric functions.

\begin{theorem}
Let $x,y,z\in(0,1)$ be such that $z^2=x^2+y^2.$ If $p\geq1,$ then the following Gr\"unbaum type inequalities are true
$$1+\frac{{\rm arcsin}_p(z^2)}{z^2}\geq\frac{{\rm arcsin}_p(x^2)}{x^2}+\frac{{\rm arcsin}_p(y^2)}{y^2},$$
$$1+\frac{{\rm arctanh}_p(z^2)}{z^2}\geq\frac{{\rm arctanh}_p(x^2)}{x^2}+\frac{{\rm arctanh}_p(y^2)}{y^2}.$$
Moreover, if $p\geq 2,$ then we have
$$1+\frac{{\rm arctan}_p(z^2)}{z^2}\leq\frac{{\rm arctan}_p(x^2)}{x^2}+\frac{{\rm arctan}_p(y^2)}{y^2},$$
$$1+\frac{{\rm arcsinh}_p(z^2)}{z^2}\leq\frac{{\rm arcsinh}_p(x^2)}{x^2}+\frac{{\rm arcsinh}_p(y^2)}{y^2},$$
and the last inequality is reversed when $p\in(0,1].$
\end{theorem}

\begin{proof}[\bf Proof]
We shall apply Lemma \ref{lem00}. Since
$$f(x)=\frac{1}{x}\left[\frac{{\rm arcsin}_p(x)}{x}-1\right]=\frac{F\left(\frac{1}{p},\frac{1}{p};1+\frac{1}{p};x^p\right)-1}{x}=
\sum_{n\geq1}\frac{\frac{1}{p}\left(\frac{1}{p},n\right)}{\frac{1}{p}+n}\frac{x^{pn-1}}{n!},$$
$$g(x)=\frac{1}{x}\left[\frac{{\rm arctanh}_p(x)}{x}-1\right]=\frac{F\left(1,\frac{1}{p};1+\frac{1}{p};x^p\right)-1}{x}=
\sum_{n\geq1}\frac{\frac{1}{p}}{\frac{1}{p}+n}{x^{pn-1}},$$
it is clear that the functions $f$ and $g$ are increasing on $(0,1)$ for all $p\geq1.$ Applying Lemma \ref{lem00} these imply the first two inequalities of the theorem.

Now, consider the functions $u,v,w,q:(0,1)\to \mathbb{R},$ defined by
$$u(x)=\frac{1}{x}\left[\frac{{\rm arctan}_p(x)}{x}-1\right],\ \ \  v(x)=2\cdot{\rm arctan}_p(x)-x\left(1+\frac{1}{1+x^p}\right),$$
$$w(x)=\frac{1}{x}\left[\frac{{\rm arcsinh}_p(x)}{x}-1\right],\ \ \  q(x)=2\cdot{\rm arcsinh}_p(x)-x\left(1+\frac{1}{(1+x^p)^{\frac{1}{p}}}\right).$$
Since $$u'(x)=-\frac{v(x)}{x^3}, \ v'(x)=\frac{x^p(p-1-x^p)}{(1+x^p)^2},\ w'(x)=-\frac{q(x)}{x^3},\ q'(x)=\frac{1+2x^p-(1+x^p)^{\frac{1}{p}+1}}{(1+x^p)^{\frac{1}{p}+1}},$$
it follows that $v$ is increasing for $p\geq 2$ and thus $v(x)>\lim_{x\to 0}v(x)=0,$ that is, we have
$$2\cdot{\rm arctan}_p(x)>x\left(1+\frac{1}{1+x^p}\right)$$
for all $p\geq 2$ and $x\in(0,1).$ Consequently $u$ is decreasing for $p\geq 2$ and applying Lemma \ref{lem00} the third Gr\"unbaum inequality of this theorem follows. Finally, by applying the Bernoulli inequality \cite[p. 34]{mitri} $(1+y)^a<1+ay,$ where $y>-1,$ $y\neq0$ and $a\in(0,1),$ it follows that
$$1+2x^p-(1+x^p)(1+x^p)^{\frac{1}{p}}>1+2x^p-(1+x^p)\left(1+\frac{x^p}{p}\right)=\left(1-\frac{1}{p}-\frac{x^p}{p}\right)x^p>0$$
for $x\in(0,1)$ and $p\geq2.$ Consequently, $q$ is increasing and hence $q(x)>\lim_{x\to 0}q(x)=0$, that is, we have for $x\in(0,1)$ and $p\geq 2$
$$2\cdot{\rm arcsinh}_p(x)>x\left(1+\frac{1}{(1+x^p)^{\frac{1}{p}}}\right).$$
Hence $w$ is decreasing for $p\geq 2$ and applying Lemma \ref{lem00} the last Gr\"unbaum inequality of this theorem follows. Moreover, by using the corresponding Bernoulli inequality \cite[p. 34]{mitri} $(1+y)^a>1+ay,$ where $y>-1,$ $y\neq0$ and $a>1,$ it follows that
$$1+2x^p-(1+x^p)^{\frac{1}{p}+1}<1+2x^p-\left(1+\left(\frac{1}{p}+1\right)x^p\right)=\left(1-\frac{1}{p}\right)x^p\leq0$$
for $x\in(0,1)$ and $p\in(0,1].$ Consequently, $q$ is decreasing and hence $q(x)<\lim_{x\to 0}q(x)=0$. Hence $w$ is increasing for $p\in(0,1]$ and applying Lemma \ref{lem00} we obtain that the last Gr\"unbaum inequality of this theorem is reversed.
\end{proof}

Finally, our aim is to present some new lower and upper bounds for the functions ${\rm arctanh}_p$ and ${\rm arctan}_p.$

\begin{theorem}\label{tibthm} For all $p>1$, $x\in (0,1)$ there holds
   \begin{align} \label{I1}
      {\rm arctanh}_p(x) &< \frac x2 \left( 1 - \frac2p\, \log\big(1-x^{\frac p2}\big)
         + \frac{2^{\frac2p}\,b_{\frac p2}}{(1+x^{\frac p2})^{\frac2p}}\right)\\ \label{I2}
      {\rm arctan}_p(x) &< x\left( 1- \frac1{p(1+p)}\, \log(1-x^p)
         - \frac1p\,\log (1+x^p) \right) =: R_p(x)\, ,
   \end{align}
where
   \[ b_s := \frac1{2s}\left\{ \psi\left( \frac{1+s}{2s}\right)
           - \psi\left(\frac1{2s}\right)\right\}\, .\]
Moreover, we have
   \begin{equation} \label{I3}
      {\rm arctanh}_p(x) > \frac x2 \left( 1 - \frac2{2+p}\, \log\big(1-x^{\frac p2}\big)
                        + \frac{p (2+p)(1+x^{\frac p2}) + 4x^{\frac p2}}
                          {p (2+p)(1+x^{\frac p2})^{1+\frac2p}}
                          \right)\, ,
   \end{equation}
and
   \begin{equation} \label{I4}
      {\rm arctan}_p(x) > x \left( 1 + \frac1{p(1+p)}\, \log\big(1-x^p\big)
                        - \frac2{1+2p}\, \log(1+x^p)\right) =: L_p(x).
   \end{equation}
\end{theorem}

\begin{proof}[\bf Proof]
Since
   \[  \frac2{1-t^{2p}} = \frac1{1-t^p} + \frac1{1+t^p}\, ,\]
integrating between $0$ and $x$, we get
   \begin{equation} \label{I5}
      2\,{\rm arctanh}_{2p}(x) = {\rm arctanh}_p(x) + {\arctan}_p(x)\, .
   \end{equation}
Being \cite[Theorem 1.2]{bvarxiv}
   \begin{equation} \label{I6}
      x\left(1 - \frac1{1+p}\, \log(1-x^p)\right) < {\rm arctanh}_p(x) < x\left(1-\frac1p\,
                    \log\left(1-x^p\right)\right)\, ,
   \end{equation}
and simultaneously \cite[Theorem 1.1]{bvarxiv}
   \begin{equation} \label{I7}
      \frac{(p(1+p)(1+x^p)+x^p)x}{p(1+p)(1+x^p)^{1+\frac1p}} <
               {\arctan}_p(x) < \frac{2^pb_p\, x}{(1+x^p)^{\frac1p}}\, ,
   \end{equation}
taking throughout $2p \mapsto p$, the rest is straightforward. Also, the lower bound
in \eqref{I3} follows by making use of both lower bounds in the previous bilateral  inequalities.

The upper bound \eqref{I2} we conclude by
   \[ {\arctan}_p(x) = 2\,{\rm arctanh}_{2p}(x) -{\rm arctanh}_p(x) \, ,\]
since the upper bound for the inverse hyperbolic and the lower bound for the inverse trigonometric functions yield the statement. The associated lower bound \eqref{I4} we achieve
by applying the lower bound for ${\rm arctanh}_{2p}$ and the upper bound for
${\rm arctanh}_{p}$ in \eqref{I6} and \eqref{I7}, respectively.

Now it remains to prove that the inequality pairs \eqref{I1} \& \eqref{I3} and
\eqref{I2} \& \eqref{I4} are not redundant. For instance the difference
of the upper bound in \eqref{I2} and the lower bound in \eqref{I4} has to be positive for all
$p>1$ for certain fixed $x \in (0,1)$. Indeed, being
   \begin{align*}
      R_p(x) - L_p(x) &= -\frac2{p(1+p)}\, \log(1-x^p) - \frac1{p(1+2p)}\, \log(1+x^p) \\
             &> \frac2{p(1+p)}\,x^p - \frac1{p(1+2p)}\,x^p = \frac{x^p}{(1+p)(1+2p)}>0\, .
   \end{align*}
The proof is complete.
\end{proof}


\section{\bf Bilateral inequalities for generalized hypergeometric ${}_3F_2$ function}
\setcounter{equation}{0}

In the sequel we deduce some lower and upper bounds for two different kind of generalized hypergeometric ${}_3F_2$ (or sometimes called Clausen's) functions, by bounding {\em mutatis mutandis} both sum and difference $\arcsin_p(x)\pm {\rm arcsinh}_p(x)$, by already known  bilateral bounds. The generalized hypergeometric  ${}_3F_2$ function, is defined by the power series
   \[ {}_3F_2\left( a,\,b,\,c ;\, d,\,e ; z \right)
           = \sum_{n \geq 0} \frac{(a,n)(b,n)(c,n)}{(d,n)(e,n)}\, \frac{z^n}{n!}\, ,\]
and converges for all $|z|<1$.

The first main result of this section reads as follows.

\begin{theorem}\label{th3F2} For all $a \in (0, \tfrac12)$ and $x \in (0,1)$ we have
   \begin{equation} \label{J1}
      L_a(x) < {}_3F_2\left( a, a, a+\tfrac12; \tfrac12, a+1 ; x^{\frac1a} \right) < R_a(x),
   \end{equation}
where
   \begin{align*}
      L_a(x) &= \frac12\,\left(1 + \frac{4a^2}{1+2a}\,x^{\frac1{2a}}\right)
           + \frac{1+\frac{2a}{1+2a}\log(1+x^{\frac1{2a}})}{2(1+x^{\frac1{2a}})^{2a}}\\
      R_a(x) &= \frac{\pi_{\frac1{2a}}}4
              + \frac{1+2a \log(1+x^{\frac1{2a}})}{2(1+x^{\frac1{2a}})^{2a}} ,
   \end{align*}
and
   \[ \pi_p = \frac{2\pi}{p\, \sin(\pi/p)}\, .\]
\end{theorem}

\begin{proof}[\bf Proof]
Consider the sum
   \[ \arcsin_p(x) + {\rm arcsinh}_p(x) = x\left[ F\left(\frac{1}{p},\frac{1}{p};
             1+\frac{1}{p};x^p\right) + F\left(\frac{1}{p},\frac{1}{p};
             1+\frac{1}{p};-x^p\right)\right]\, .\]
Summing up the hypergeometric functions, after some routine calculations we get
   \begin{equation} \label{J3}
      \arcsin_p(x) + {\rm arcsinh}_p(x) = 2x\,{}_3F_2\left(\frac1{2p},\,
             \frac1{2p},\,\frac1{2p}+\frac12; \frac12,\,\frac1{2p}+1; x^{2p} \right)\, .
   \end{equation}
Indeed, starting with
   \begin{align*}
      F\left(\frac1p, \frac1p; 1+\frac{1}{p};x^p \right)
        &+ F\left(\frac{1}{p},\frac{1}{p}; 1+\frac{1}{p};-x^p \right) \\
        &\hspace{-10mm} = \sum_{n \geq 0}
           (1+(-1)^n)\,\frac{(\frac1p,n)\,(\frac1p,n)}{(\frac1p+1,n)}\,
           \frac{(x^p)^n}{n!}
         = 2 \sum_{n \geq 0} \frac{(\frac1p,2n)\,(\frac1p,2n)}{(\frac1p+1,2n)}\,
           \frac{(x^p)^{2n}}{(2n)!} := H_1\, ,
   \end{align*}
and following the appropriate form of the Legendre duplication formula, which reads
   \[ \big(2s\big)_{2n} = 4^n\,\big(s,n\big)\,\big(s+\tfrac12,n\big)\,, \]
we deduce the desired statement
   \begin{align*}
      H_1 &= 2 \sum_{n \geq 0} \frac{(\frac1{2p},n)^2\,(\frac1{2p}+\frac12,n)^2}
             {(\frac1{2p}+\tfrac12,n)\,(\frac1{2p}+1,n)\,
             (\frac12,n)}\,\frac{(x^{2p})^n}{(1,n)}\\
          &= 2\cdot {}_3F_2\left(\frac1{2p},\, \frac1{2p},\,\frac1{2p}+\frac12;
             \frac12,\,\frac1{2p}+1; x^{2p} \right) .
   \end{align*}
Now, having in mind that \cite[Theorem 1.1]{bvarxiv}
   \begin{equation} \label{J4}
      \left(1 + \frac{x^p}{p(1+p)}\right)\,x < \arcsin_p(x) < \frac{\pi_p}2\,x\, ,
   \end{equation}
and respectively \cite[Theorem 1.2]{bvarxiv}
   \begin{equation} \label{J5}
      \frac{x(1+\frac1{1+p}\log(1+x^p))}{(1+x^p)^{\frac1p}} < {\rm arcsinh}_p(x) <
      \frac{x(1+\frac1p \log(1+x^p))}{(1+x^p)^{\frac1p}} ,
   \end{equation}
we deduce {\em via} \eqref{J3} the upper bound
   \[ {}_3F_2\left( \begin{array}{c} \frac1{2p},\,\frac1{2p},\,\frac1{2p}+\frac12 \\
             \frac12,\,\frac1{2p}+1 \end{array}; x^{2p} \right) <
      \frac{\pi_p}4 + \frac{1+\frac1p \log(1+x^p)}{2(1+x^p)^{\frac1p}};\]
while the related lower bound becomes
   \[ {}_3F_2\left( \begin{array}{c} \frac1{2p},\,\frac1{2p},\,\frac1{2p}+\frac12 \\
             \frac12,\,\frac1{2p}+1 \end{array}; x^{2p} \right) >
             \frac12\,\left(1 + \frac{x^p}{p(1+p)}\right)
           + \frac{1+\frac1{1+p}\log(1+x^p)}{2(1+x^p)^{\frac1p}}\, .\]
\end{proof}

The next result is the analogous of Theorem \ref{th3F2}.

\begin{theorem} For all $a \in (0, \tfrac12)$ and $x \in (0,1)$ we have
   \begin{equation} \label{J6}
      \widetilde L_a(x) < {}_3F_2\left( a+\tfrac12, a+\tfrac12, a+1; \tfrac32, a+\tfrac32 ;
                          x^{\frac1a} \right) < \widetilde R_a(x),
   \end{equation}
where
   \begin{align*}
      \widetilde L_a(x) &= \frac12 + \frac{1+2a}{8a^2x^{\frac1{2a}}}\,\left( 1
           - \frac{1+2a \log(1+x^{\frac1{2a}})}{(1+x^{\frac1{2a}})^{2a}} \right)\, , \\
      \widetilde R_a(x) &= \frac{1+2a}{8a^2x^{\frac1{2a}}}\,
             \left(\frac{\pi_{\frac1{2a}}}2
           - \frac{1+\frac{2a}{1+2a}\log(1+x^{\frac1{2a}})}{(1+x^{\frac1{2a}})^{2a}}\right).
   \end{align*}
\end{theorem}

\begin{proof}[\bf Proof]
We start with the difference
   \[ \arcsin_p(x) - {\rm arcsinh}_p(x) = x\left[ F\left(\frac{1}{p},\frac{1}{p};
             1+\frac{1}{p};x^p\right) - F\left(\frac{1}{p},\frac{1}{p};
             1+\frac{1}{p};-x^p\right)\right] =: x \cdot H_2\, ,\]
which becomes to the another hypergeometric type function
   \[ H_2 = \sum_{n \geq 0} (1-(-1)^n)\,\frac{(\frac1p,n)\,(\frac1p,n)}{(\frac1p+1,n)}\,
            \frac{(x^p)^n}{n!}
          = 2x^p \sum_{n \geq 0} \frac{(\frac1p,2n+1)\,(\frac1p,2n+1)}{(\frac1p+1,2n+1)}\,
            \frac{(x^p)^{2n}}{(2n+1)!}\, .\]
The formula
   \[ \big(2s\big)_{2n+1} = \frac{\Gamma(2s+1+2n)}{\Gamma(2s+2n)}\,
                            \frac{\Gamma(2s+2n)}{\Gamma(2s)}
                          = s\,2^{2n+1}\,\big(s+\tfrac12,n\big)\,\big(s+1,n\big)\,, \]
is the keystone in rewriting the sum $H_2$ into
   \[ H_2 = \frac{2x^p}{p(1+p)} \sum_{n \geq 0} \frac{(\frac{1+p}{2p},n)\,
            (\frac{1+p}{2p},n)\,(\frac1{2p}+1,n)}{(\frac32,n)\,
            (\frac1{2p}+\frac32,n)} \frac{(x^{2p})^{n}}{(1,n)}\, ,\]
which means
   \begin{equation} \label{J7}
      \arcsin_p(x) - {\rm arcsinh}_p(x) = \frac{2x^{p+1}}{p(1+p)}\,
             {}_3F_2\left(\frac{1+p}{2p},\,\frac{1+p}{2p},\,\frac1{2p}+1;
             \frac32,\,\frac1{2p}+\frac32; x^{2p} \right)\, .
   \end{equation}
Now, again by \eqref{J4} and \eqref{J5} we estimate the $\arcsin_p, {\rm arcsinh}_p$ functions
from above, getting the upper bound
   \[ \arcsin_p(x) - {\rm arcsinh}_p(x)< x\, \left(\frac{\pi_p}2 -
                     \frac{1+\frac1{1+p}\log(1+x^p)}{(1+x^p)^{\frac1p}} \right)\, ,\]
therefore
   \[ {}_3F_2\left(\frac{1+p}{2p},\,\frac{1+p}{2p},\,\frac1{2p}+1;
             \frac32,\,\frac1{2p}+\frac32; x^{2p} \right)
           < \frac{p(p+1)}{2x^p}\left( \frac{\pi_p}2
           - \frac{1+\frac1{1+p}\log(1+x^p)}{(1+x^p)^{\frac1p}}\right)\, .\]
By similar procedure we deduce the lower bound which equals
   \[ \frac{p(1+p)}{2x^p}\,\left(1 + \frac{x^p}{p(1+p)}
           - \frac{1+\frac1p \log(1+x^p)}{(1+x^p)^{\frac1p}}\right)\, .  \]
Substituting in both bounds $p = \tfrac1{2a}$, we achieve the statement of the Theorem.
\end{proof}

\end{document}